\documentclass[a4paper,11pt]{article}

\usepackage[utf8]{inputenc}
\usepackage[UKenglish]{babel}
\usepackage[T1]{fontenc}
\usepackage{amsmath}
\usepackage{amsfonts}
\usepackage{amssymb}
\usepackage{amsthm}
\usepackage{graphicx}
\usepackage[top=2cm,bottom=3cm,left=2.5cm,right=2.5cm]{geometry}
\usepackage{setspace}
\usepackage{xcolor}
\usepackage{stmaryrd}
\usepackage{enumitem}
\usepackage{esint}
\usepackage{centernot}
\usepackage{subfigure}

\usepackage[colorlinks=true]{hyperref}

\author{Guillaume Blanc}
\title{Percolation with invariant Poisson processes of lines in the $3$-regular tree}

\begin{document}

\theoremstyle{definition}
\newtheorem{prop}{Proposition}
\newtheorem{claim}{Claim}

\theoremstyle{plain}
\newtheorem{lem}{Lemma}
\newtheorem{thm}{Theorem}
\newtheorem*{thm*}{Theorem}

\theoremstyle{remark}
\newtheorem{rem}{Remark}

\maketitle

\begin{abstract}
In this paper, we study invariant Poisson processes of lines (i.e, bi-infinite geodesics) in the $3$-regular tree.
More precisely, there exists a unique (up to multiplicative constant) locally finite Borel measure on the space of lines that is invariant under graph automorphisms, and we consider two Poissonian ways of playing with this invariant measure.
First, following Benjamini, Jonasson, Schramm and Tykesson, we consider an invariant Poisson process of lines, and show that there is a critical value of the intensity below which a.s.~the vacant set of the process percolates, and above which all its connected components are finite.
Then, we consider an invariant Poisson process of roads (i.e, lines with speed limits), and show that there is a critical value of the parameter governing the speed limits of the roads below which a.s.~one can drive to infinity in finite time using the road network generated by the process, and above which this is impossible.
\end{abstract}

\begin{figure}[ht]
\centering
\includegraphics[width=0.5\linewidth]{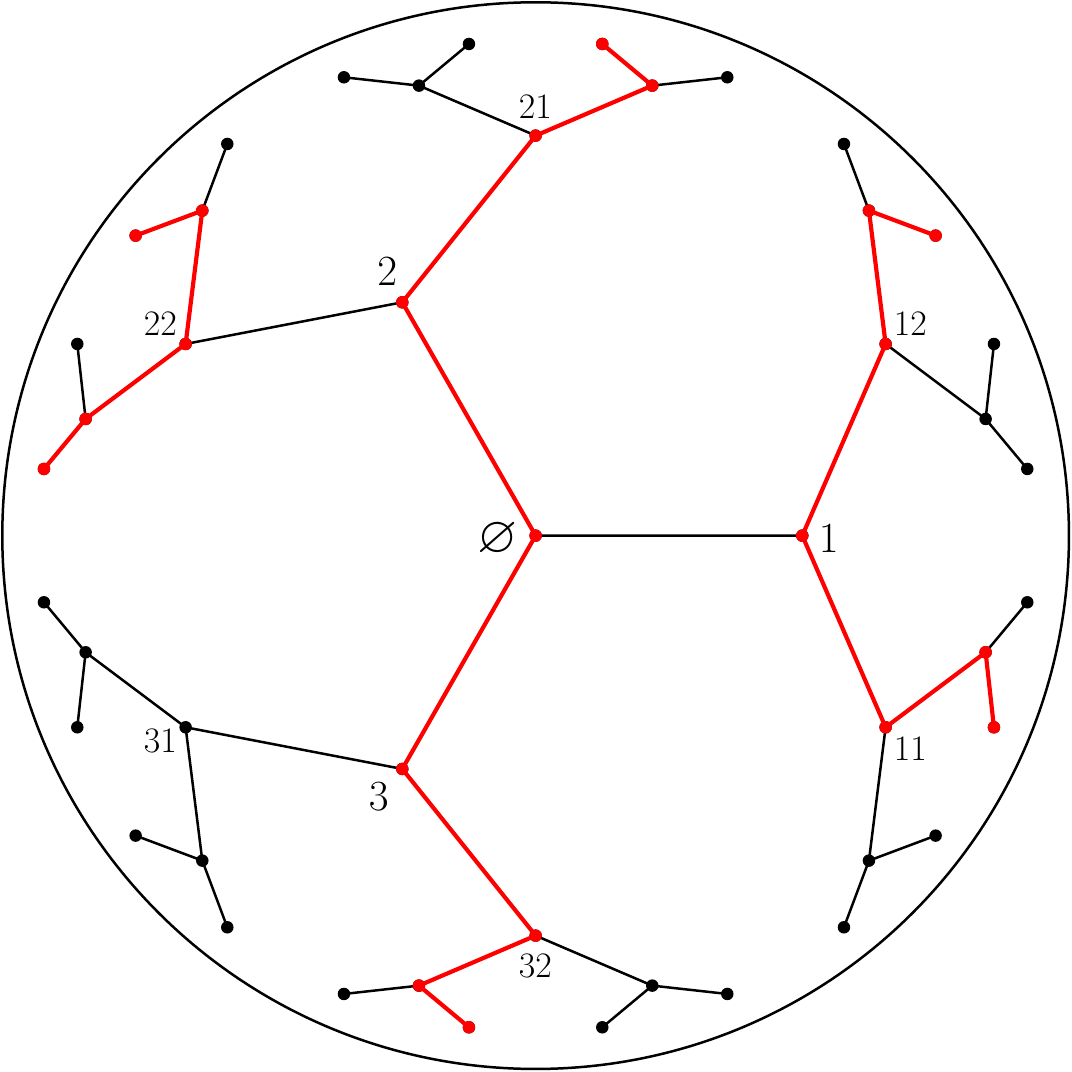}
\caption{The $3$-regular tree $\mathbb{T}$ (planar, rooted, and labeled using the Neveu notation, or Ulam--Harris labelling).
In red, we have represented a few lines in $\mathbb{T}$.}\label{fig:tree}
\end{figure}

\section*{Introduction and main results}

Let $\mathbb{T}$ be the $3$-regular tree (planar, rooted, and labeled using the Neveu notation, or Ulam--Harris labelling; as represented in Figure \ref{fig:tree}), and let $\mathbb{L}$ be the space of lines (i.e, bi-infinite geodesics) in $\mathbb{T}$.
There exists a unique (up to multiplicative constant) locally finite Borel measure on $\mathbb{L}$ that is invariant under graph automorphisms.
While this is certainly guaranteed by abstract results on Haar measures (see, e.g, \cite[Chapter 13]{stochintgeo} and references therein), we give a simple description of $\mu$ in terms of the uniform measure on the boundary of $\mathbb{T}$ (see Proposition \ref{prop:mu} below), which is analogous to the better-known case of the hyperbolic plane \cite[Section 6]{visibility}.
Then, we consider two Poissonian ways of playing with this invariant measure.
To give some extra motivation, let us take a step back and abstract the setting a little bit (while remaining informal).

\paragraph{Motivation.}
Picture a nice homogeneous metric space $X$, on which a group of isometries acts transitively.
Add to the picture a class $\mathcal{F}$ of closed subsets of $X$, stable under the action of the isometries, and assume that $\mathcal{F}$ is equipped with an isometry-invariant measure $\mu$.
For instance, if $X$ itself is equipped with an isometry-invariant Borel measure $\lambda$, then the class ${\mathcal{F}=\left\{\overline{B}(x,r)\,;\,x\in X\right\}}$ of closed balls with some fixed radius is stable under the action of the isometries, and the pushforward $\mu$ of $\lambda$ by the map $x\in X\mapsto\overline{B}(x,r)$ is an isometry-invariant measure on $\mathcal{F}$.
Consider the following concrete examples.
\begin{enumerate}[label=(\roman*)]
\item $X$ is the Euclidean lattice $\mathbb{Z}^d$, equipped with the counting measure $\lambda$, and 
\[\mathcal{F}=\left\{\{x\}\,;\,x\in\mathbb{Z}^d\right\}.\]
\item $X$ is the Euclidean space $\mathbb{R}^d$, equipped with the Lebesgue measure $\lambda$, and 
\[\mathcal{F}=\left\{\overline{B}(x,1)\,;\,x\in\mathbb{R}^d\right\}.\]
\newcounter{i}
\setcounter{i}{\value{enumi}}
\end{enumerate}
In settings $(X,\mathcal{F},\mu)$ such as described above, two separate problems may be considered: a percolation problem, and a driving distance problem.
\begin{itemize}
\item \textbf{Percolation.} One can take a Poisson process $\Pi$ with intensity $\alpha\cdot\mu$ on $\mathcal{F}$, where $\alpha>0$ is a parameter, and ask about the percolative properties of the trace $\bigcup_{F\in\Pi}F$ of the process; or that of its complement, the vacant set $\mathcal{V}=\left.X\middle\backslash\bigcup_{F\in\Pi}F\right.$.
\item \textbf{Driving distance.} One can take a Poisson process $\Pi$ with intensity $\mu\otimes v^{-\beta}\mathrm{d}v$ on $\mathcal{F}\times\mathbb{R}_+^*$, where $\beta>1$ is a parameter; and, viewing each atom $(F,v)$ of $\Pi$ as a ``road'' in $X$, with $v$ the speed limit on the subset $F$, consider the random metric ${T:X\times X\rightarrow\mathbb{R}_+}$ induced by the driving distance with respect to the road network generated by $\Pi$.
To be more precise, first consider the (random) speed limit function $V:X\rightarrow\mathbb{R}_+$ defined by
\[V(x)=\sup\{v\,;\,(F,v)\in\Pi:F\ni x\}\quad\text{for all $x\in X$,}\]
with the convention $\sup\emptyset=0$.
Then, define the driving distance $T(x,y)$ between $x,y\in X$ as the infimal time $T>0$ for which there exists a path $\gamma:[0,T]\rightarrow X$ from $x$ to $y$ that respects the speed limits set by $\Pi$, in the sense that
\[d(\gamma(s),\gamma(t))\leq\int_s^tV(\gamma(u))\mathrm{d}u\quad\text{for all $s,t\in[0,T]$.}\]
Equivalently, the driving distance metric $T:X\times X\rightarrow\mathbb{R}_+$ is the first passage percolation distance function associated with the random field $W:x\in X\mapsto1/V(x)$.
\end{itemize}
Note that in the case of example (i) given above, the percolation problem amounts to Bernoulli site percolation on the Euclidean lattice, while the driving distance problem broadly amounts to site first passage percolation.
In the case of example (ii), the percolation problem amounts to the continuum percolation model known as the Gilbert disk (or Boolean) model.
That being said, one can also imagine settings $(X,\mathcal{F},\mu)$ allowing long range models, where the invariant measure $\mu$ on $\mathcal{F}$ does not simply come from an invariant measure $\lambda$ on $X$.
We have in mind the following examples.
\begin{enumerate}[label=(\roman*)]
\setcounter{enumi}{\value{i}}
\item $X$ is the Euclidean space $\mathbb{R}^d$ ($d\geq2$) and $\mathcal{F}$ is the space of affines lines, equipped with its unique (up to multiplicative constant) locally finite invariant Borel measure $\mu$.

In this case, the driving distance problem has been introduced by Aldous \cite{aldous} and Kendall \cite{kendall} a few years ago.
It has been shown \cite{kendall,kahn,moa} that the driving distance metric ${T:\mathbb{R}^d\times\mathbb{R}^d\rightarrow\mathbb{R}_+}$ is well-defined for and only for $\beta>d$, and that the random metric space $\left(\mathbb{R}^d,T\right)$ is homeomorphic to the Euclidean space $\mathbb{R}^d$ and has Hausdorff dimension $(\beta-1)d/(\beta-d)>d$.
\item $X$ is the hyperbolic plane $\mathbb{H}$ and $\mathcal{F}$ is the space of lines (i.e, bi-infinite geodesics), equipped with its unique (up to multiplicative constant) locally finite invariant Borel measure $\mu$.

In this case, the percolation problem has been considered by Benjamini, Jonasson, Schramm and Tykesson \cite{visibility}.
They have shown the existence of a critical parameter $\alpha_0=1$ (this explicit value depends on the normalisation of $\mu$) such that the vacant set $\mathcal{V}$ contains lines for $\alpha<\alpha_0$, and does not contain any half-line for $\alpha\geq\alpha_0$. 
\item $X$ is the Euclidean lattice $\mathbb{Z}^d$ ($d\geq3$) and $\mathcal{F}$ is the space of bi-infinite transient paths, equipped with the invariant measure $\mu$ constructed by Sznitman \cite{sznitman}.

In this case, the percolation problem amounts to the random interlacements model introduced by Sznitman \cite{sznitman}, for which it has been shown \cite{sznitman,sidoravicius} that there exists a critical parameter $\alpha_0\in\mathbb{R}_+^*$ such that the vacant set $\mathcal{V}$ percolates for $\alpha<\alpha_0$, and has only finite connected components for $\alpha>\alpha_0$.
\end{enumerate}

\paragraph{Our setting.}
In this paper, we consider the two problems (percolation and driving distance) in the case where $X$ is the $3$-regular tree $\mathbb{T}$ and $\mathcal{F}$ is the space of lines $\mathbb{L}$, equipped with its unique (up to multiplicative constant) locally finite invariant Borel measure $\mu$.
As it turns out, a normalisation for $\mu$ can be specified by asking that for every $x\neq y\in\mathbb{T}$,
\[\mu\{\ell\in\mathbb{L}:\text{$\ell$ passes through $x$ and $y$}\}=2^{-d(x,y)},\]
where $d(x,y)$ is the graph distance between $x$ and $y$ in $\mathbb{T}$.

\subparagraph{Percolation (visibility to infinity, despite obstacles).}
Following Benjamini, Jonasson, Schramm and Tykesson \cite{visibility}, we let $\Pi$ be a Poisson process with intensity $\alpha\cdot\mu$ on $\mathbb{L}$, where $\alpha>0$ is a parameter.
We recover their result \cite[Proposition 6.1]{visibility} in this discrete setting.

\begin{thm}\label{thm:visibility}
There exists a critical parameter $\alpha_0=4\ln2$ (this explicit value depends on the normalisation of $\mu$ specified above) such that the following holds.
\begin{itemize}
\item For $\alpha<\alpha_0$, almost surely, the vacant set $\mathcal{V}$ contains a line.
\item For $\alpha\geq\alpha_0$, almost surely, the vacant set $\mathcal{V}$ does not contain any half-line.
\end{itemize}
\end{thm}
We present this in Section \ref{sec:lines}.

\subparagraph{Driving distance.}
In Section \ref{sec:roads}, which is the core of the paper, we le $\Pi$ be a Poisson process with intensity proportional to $\mu\otimes v^{-\beta}\mathrm{d}v$ on $\mathbb{L}\times\mathbb{R}_+^*$, where $\beta>1$ is a parameter.
Viewing each atom $(\ell,v)$ of $\Pi$ as a road in $\mathbb{T}$, with $v$ the speed limit on the line $\ell$, we consider the random metric $T:\mathbb{T}\times\mathbb{T}\rightarrow\mathbb{R}_+$ induced by the driving distance with respect to the road network generated by $\Pi$ (see Section \ref{sec:roads} for a more detailed presentation of the model).
In this instance of first passage percolation with positively associated passage times, we prove that the so-called explosion phenomenon undergoes a phase transition in terms of the parameter $\beta$.
This is the main novel result of this paper.

\begin{thm}\label{thm:explosion}
The explosion phenomenon undergoes a phase transition at $\beta=2$.
\begin{itemize}
\item For $\beta<2$, almost surely, there exists an infinite geodesic path ${(\varnothing=x_0,x_1,\ldots)}$ in $\mathbb{T}$ such that ${\sum_{n\geq1}T(x_{n-1},x_n)<\infty}$.
\item For $\beta>2$, almost surely, for every infinite geodesic path ${(\varnothing=x_0,x_1,\ldots)}$ in $\mathbb{T}$, we have ${\sum_{n\geq1}T(x_{n-1},x_n)=\infty}$.
\end{itemize}
\end{thm}

\paragraph{Acknowledgements.}
I warmly thank Nicolas Curien and Arvind Singh for their constant support and guidance, and for their valuable comments on earlier versions of this paper.
We thank Itai Benjamini for suggesting to look at the driving distance problem in the $3$-regular tree.
Finally, I am grateful to the PizzaMa team in Orsay for their encouraging feedback.

\tableofcontents

\section{The invariant measure on the space of lines}\label{sec:mu}

In this section, we give a description of the invariant measure $\mu$ on the space of lines $\mathbb{L}$ in terms of the uniform measure on the boundary of $\mathbb{T}$ (see Proposition \ref{prop:mu} below).
This is analogous to better-known case of the hyperbolic plane (see, e.g, \cite[Section 6]{visibility}), and certainly not new, but we were unable to find it in the literature.
Let us first recall some basic definitions and facts.
We rely on the authoritative reference \cite{lyonsperes}.

\begin{itemize}
\item A \emph{line} in $\mathbb{T}$ is the trace $\{x_n,\,n\in\mathbb{Z}\}$ of a bi-infinite geodesic path $(x_n)_{n\in\mathbb{Z}}$, i.e, such that $d(x_m,x_n)=|m-n|$ for all $m,n\in\mathbb{Z}$, where $d(\cdot,\cdot)$ denotes the graph distance on $\mathbb{T}$.
We denote by $\mathbb{L}$ the set of lines.

\item A \emph{ray} in $\mathbb{T}$ is an infinite non-backtracking path $(\varnothing=x_0,x_1,\ldots)$ starting at the root. 
We denote by $\partial\mathbb{T}$ the set of rays, also known as the boundary of $\mathbb{T}$.
Given two distinct rays $\xi=(x_n)_{n\in\mathbb{N}}$ and $\eta=(y_n)_{n\in\mathbb{N}}$, we denote by $\xi\wedge\eta$ the farthest node from the root that is common to both paths $\xi$ and $\eta$.
We equip $\partial\mathbb{T}$ with the metric $d:\partial\mathbb{T}\times\partial\mathbb{T}\rightarrow\mathbb{R}_+$ defined by 
\[d(\xi,\eta)=\begin{cases}
0&\text{if $\xi=\eta$}\\
2^{-|\xi\wedge\eta|}&\text{otherwise}
\end{cases}\quad\text{for all $\xi,\eta\in\partial\mathbb{T}$}.\]
There is a natural Borel probability measure on $\partial\mathbb{T}$; namely, the law of a non-backtracking random walk $(X_n)_{n\in\mathbb{N}}$ starting at the root.
With this measure as mass distribution, it is not difficult to check that $(\partial\mathbb{T},d)$ has Hausdorff dimension $1$.
\item Given two distinct rays ${\xi=(x_n)_{n\in\mathbb{N}}}$ and ${\eta=(y_n)_{n\in\mathbb{N}}}$, we denote by $\Lambda(\xi,\eta)$ the line with endpoints $\xi$ and $\eta$; namely,
\[\Lambda(\xi,\eta)=\{\ldots,x_{n+1},\xi\wedge\eta,y_{n+1},\ldots\},\]
where $n=|\xi\wedge\eta|$.
Denoting by ${\Delta=\{(\xi,\xi)\,;\,\xi\in\partial\mathbb{T}\}}$ the diagonal in $\partial\mathbb{T}\times\partial\mathbb{T}$, consider the surjective (two-to-one) mapping
\[\begin{matrix}
\Lambda:&(\partial\mathbb{T}\times\partial\mathbb{T})\setminus\Delta&\longrightarrow&\mathbb{L}\\
&(\xi,\eta)&\longmapsto&\Lambda(\xi,\eta).
\end{matrix}\]
We endow $\mathbb{L}$ with the finest topology that makes $\Lambda$ into a continuous map (final topology), and with the corresponding Borel $\sigma$-algebra.
\item Every graph automorphism ${\phi:\mathbb{T}\rightarrow\mathbb{T}}$ naturally extends to a continuous map from $\mathbb{L}$ to $\mathbb{L}$, mapping the line ${\{x_n,\,n\in\mathbb{Z}\}}$ onto $\{\phi(x_n),\,n\in\mathbb{Z}\}$.
We say that a Borel measure $\mu$ on $\mathbb{L}$ is \emph{invariant} if for every graph automorphism $\phi:\mathbb{T}\rightarrow\mathbb{T}$, the pushforward $\phi_*\mu$ of $\mu$ by $\phi$ agrees with $\mu$.
\item For a subset $S\subset\mathbb{T}$, we denote by $\langle S\rangle=\{\ell\in\mathbb{L}:\ell\cap S\neq\emptyset\}$ the set of lines that hit $S$.
For $x\in\mathbb{T}$, we write $\langle x\rangle=\langle\{x\}\rangle$ for the set of lines that pass through $x$; and for $x\neq y\in\mathbb{T}$, we write $\langle x,y\rangle=\langle x\rangle\cap\langle y\rangle$ for the set of lines that pass through both points $x$ and $y$.
Note that $\langle x,y\rangle$ is not the same as $\langle\{x,y\}\rangle$.

\begin{claim}\label{claimpisystem}
The collection $\mathcal{C}=\{\langle x,y\rangle\,;\;x\neq y\in\mathbb{T}\}\cup\{\emptyset\}$ forms a $\pi$-system that generates the Borel $\sigma$-algebra on $\mathbb{L}$.
\end{claim}
\begin{proof}[Sketch of proof]
The fact that $\mathcal{C}$ is a $\pi$-system is clear.
That $\mathcal{C}$ is included in the Borel $\sigma$-algebra follows from the fact that the ${(\langle x,y\rangle\,;\;x\neq y\in\mathbb{T})}$ are closed subsets of $\mathbb{L}$.
Finally, the collection $\mathcal{C}$ generates the Borel $\sigma$-algebra on $\mathbb{L}$, as every open subset of $\mathbb{L}$ can be written as a countable union of elements of $\mathcal{C}$.
\end{proof}

\item Finally, we say that a Borel measure $\mu$ on $\mathbb{L}$ is \emph{locally finite} if $\mu\langle x,y\rangle<\infty$ for all $x\neq y\in\mathbb{T}$.
\end{itemize}

We are now ready to state the main result of this section.

\begin{prop}\label{prop:mu}
The Borel measure $\mu$ defined by
\[\int_\mathbb{L}\varphi(\ell)\mathrm{d}\mu(\ell)=\mathbb{E}\left[\frac{\varphi(\Lambda(X,Y))}{d(X,Y)^2}\right]\quad\text{for all Borel functions $\varphi:\mathbb{L}\rightarrow[0,\infty]$,}\]
where $X=(X_n)_{n\in\mathbb{N}}$ and $Y=(Y_n)_{n\in\mathbb{N}}$ are two independent non-backtracking random walks starting at the root in $\mathbb{T}$, is locally finite and invariant.
Moreover, for any locally finite invariant Borel measure $\mu$ on $\mathbb{L}$, there exists a constant $c>0$ such that
\[\mu\langle x,y\rangle=c\cdot2^{-d(x,y)}\quad\text{for all $x\neq y\in\mathbb{T}$.}\]
\end{prop}
\begin{proof}
We start with the second assertion.
Let $\mu$ be a locally finite invariant Borel measure on $\mathbb{L}$, and let $x\neq y\in\mathbb{T}$.
Set ${n=d(x,y)}$, and let $1_n$ be the vertex $1\ldots1\in\mathbb{T}$ whose distance to the root is $n$.
We have $\mu\langle x,y\rangle=\mu\langle\varnothing,1_n\rangle$.
Now, we claim that ${\mu\langle\varnothing,1_n\rangle=\mu\langle\varnothing,1\rangle\cdot2^{-(n-1)}}$ for all $n\in\mathbb{N}^*$.
This is easily checked by induction: for each $n\in\mathbb{N}^*$, the set $\langle\varnothing,1_n\rangle$ can be written as the disjoint union $\langle\varnothing,1_n1\rangle\sqcup\langle\varnothing,1_n2\rangle$; thus, by invariance,
\[\mu\langle\varnothing,1_n\rangle=\mu\langle\varnothing,1_n1\rangle+\mu\langle\varnothing,1_n2\rangle=2\cdot\mu\langle\varnothing,1_{n+1}\rangle.\]
For the first assertion, let $\mu$ be the Borel measure defined by
\[\int_\mathbb{L}\varphi(\ell)\mathrm{d}\mu(\ell)=\mathbb{E}\left[\frac{\varphi(\Lambda(X,Y))}{d(X,Y)^2}\right]\quad\text{for all Borel functions $\varphi:\mathbb{L}\rightarrow[0,\infty]$,}\]
where $X=(X_n)_{n\in\mathbb{N}}$ and $Y=(Y_n)_{n\in\mathbb{N}}$ are two independent non-backtracking random walks starting at the root.
We claim that
\begin{equation}\label{eq:constructmu}
\mu\langle x,y\rangle=\frac{8}{9}\cdot2^{-d(x,y)}\quad\text{for all $x\neq y\in\mathbb{T}$.}
\end{equation}
Indeed, let $x\neq y\in\mathbb{T}$.
We have
\[\mu\langle x,y\rangle=\mathbb{E}\left[d(X,Y)^{-2}\,;\,\text{$\Lambda(X,Y)$ passes through $x$ and $y$}\right].\]
Now, we distinguish two cases.
\begin{itemize} 
\item First, suppose that $x$ and $y$ are not descendants of one another.
Then $\Lambda(X,Y)$ passes through $x$ and $y$ if and only if $X$ passes through $x$ and $Y$ through $y$, or (exclusive) $X$ passes through $y$ and $Y$ through $x$.
Moreover, on that event, we have ${d(X,Y)=2^{-|x\wedge y|}}$.
It follows that
\[\begin{split}
\mu\langle x,y\rangle&=2^{2|x\wedge y|}\cdot\left(3\cdot2^{|x|-1}\right)^{-1}\cdot\left(3\cdot2^{|y|-1}\right)^{-1}+2^{2|x\wedge y|}\cdot\left(3\cdot2^{|y|-1}\right)^{-1}\cdot\left(3\cdot2^{|x|-1}\right)^{-1}\\
&=2\cdot\frac{4}{9}\cdot2^{2|x\wedge y|-|x|-|y|}=\frac{8}{9}\cdot2^{-d(x,y)}.
\end{split}\]
\item Next, to treat the case where $x$ and $y$ are descendants of one another, we may assume without loss of generality that $x\prec y$.
Let us denote by $\varnothing=z_0,\ldots,z_n=y$ the geodesic path from the root to $y$.
By assumption, we have $z_m=x$ for some $m\in\llbracket0,n\llbracket$.
Now, we define $x_{-1}$ and $x_0$ as the two neighbours of $\varnothing$ that are not $z_1$; and for each $k\in\llbracket1,m\rrbracket$, we define $x_k$ as the neighbour of $z_k$ that is neither $z_{k-1}$ nor $z_{k+1}$.
See Figure \ref{fig:mu} for an illustration.

\begin{figure}[ht]
\centering
\includegraphics[width=0.5\linewidth]{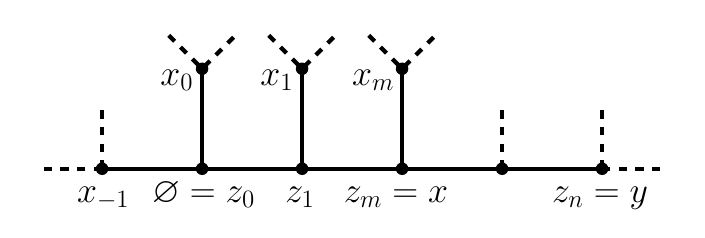}
\caption{Illustration of the definition of $x_{-1},x_0,\ldots,x_m$.}\label{fig:mu}
\end{figure}

The set $\langle x,y\rangle$ can be written as the disjoint union $\bigsqcup_{k=-1}^m\langle x_k,y\rangle$, where for each $k\in\llbracket-1,m\rrbracket$, the vertices $x_k$ and $y$ are not descendants of one another.
By the previous case, it follows that
\[\begin{split}
\mu\langle x,y\rangle&=\sum_{k=-1}^m\frac{8}{9}\cdot2^{-d(x_k,y)}\\
&=\frac{8}{9}\cdot\left(2^{-(n+1)}+\sum_{k=0}^m2^{-(n-k+1)}\right)\\
&=\frac{8}{9}\cdot\left(2^{-(n+1)}+2^{-(n+1)}\cdot\left(2^{m+1}-1\right)\right)\\
&=\frac{8}{9}\cdot2^{-(n-m)}=\frac{8}{9}\cdot2^{-d(x,y)}.
\end{split}\]
\end{itemize}
This completes the proof of \eqref{eq:constructmu}.
Now, let $\phi:\mathbb{T}\rightarrow\mathbb{T}$ be a graph automorphism.
By \eqref{eq:constructmu}, we have
\[\mu\langle\phi(x),\phi(y)\rangle=\frac{8}{9}\cdot2^{-d(\phi(x),\phi(y))}=\frac{8}{9}\cdot2^{-d(x,y)}=\mu\langle x,y\rangle\quad\text{for all $x\neq y\in\mathbb{T}$.}\]
In particular, the Borel measures $\phi_*\mu$ and $\mu$ agree on the $\pi$-system $\mathcal{C}$.
By Dynkin's $\pi$--$\lambda$ theorem, we deduce that $\mu$ is invariant.
\end{proof}

In the rest of the paper, we denote by $\mu$ the unique locally finite invariant Borel measure on $\mathbb{L}$ such that
\begin{equation}\label{eq:munorm}
\mu\langle x,y\rangle=2^{-d(x,y)}\quad\text{for all $x\neq y\in\mathbb{T}$.}
\end{equation}

\section{Visibility to infinity, despite obstacles}\label{sec:lines}

In this section, we prove Theorem \ref{thm:visibility}.
Following Benjamini, Jonasson, Schramm and Tykesson \cite{visibility}, we let $\Pi$ be a Poisson process with intensity $\alpha\cdot\mu$ on $\mathbb{L}$, where $\alpha>0$ is a parameter, and $\mu$ is the invariant measure on $\mathbb{L}$ normalised by \eqref{eq:munorm}, and we consider the percolative properties of the vacant set ${\mathcal{V}=\left.\mathbb{T}\middle\backslash\bigcup_{\ell\in\Pi}\ell\right.}$.
More precisley, let us recall the statement of Theorem \ref{thm:visibility}.
\begin{itemize}
\item For $\alpha<\alpha_0$, almost surely, the vacant set $\mathcal{V}$ contains a line.
\item For $\alpha\geq\alpha_0$, almost surely, the vacant set $\mathcal{V}$ does not contain any half-line.
\end{itemize}

Before giving the proof of Theorem \ref{thm:visibility}, let us recall some basic properties of the Poisson process of lines $\Pi$.
As usual with well-behaved Poisson processes, we view $\Pi$ sometimes as a random subset of $\mathbb{L}$, and sometimes as a random atomic measure on $\mathbb{L}$, without making the distinction.
An atomic measure on $\mathbb{L}$ is a measure of the form $\pi=\sum_{k=1}^n\delta_{\ell_k}$, with $n\in\llbracket0,\infty\rrbracket$, and $\ell_k\in\mathbb{L}$ for every $k$.
We denote by $\mathbb{M}$ the space of atomic measures on $\mathbb{L}$, equipped with the $\sigma$-algebra generated by the maps
\[\begin{matrix}
\mathbb{M}&\longrightarrow&\llbracket0,\infty\rrbracket\\
\pi&\longmapsto&\pi(B),
\end{matrix}\]
for $B$ Borel subset of $\mathbb{L}$.
By construction, the Poisson process $\Pi$ has the following invariance property: for every graph automorphism $\phi:\mathbb{T}\rightarrow\mathbb{T}$, we have $\phi_*\Pi\overset{\text{\tiny law}}{=}\Pi$, where $\phi_*\Pi$ denotes the pushforward of $\Pi$ by $\phi$.
Moreover, the Poisson process $\Pi$ is mixing: if $(\psi_n)_{n\in\mathbb{N}}$ is a sequence of graph automorphisms of $\mathbb{T}$ such that ${d(\varnothing,\psi_n(\varnothing))\rightarrow\infty}$ as $n\to\infty$, then for every bounded measurable functions $f$ and $g$ from $\mathbb{M}$ to $\mathbb{R}$, we have
\[\mathbb{E}[f(\Pi)\cdot g(\psi_n^*\Pi)]\longrightarrow\mathbb{E}[f(\Pi)]\cdot\mathbb{E}[g(\Pi)]\quad\text{as $n\to\infty$,}\]
where $\psi_n^*\Pi$ denotes the pushforward of $\Pi$ by $\psi_n$.
In particular, any invariant event has probability $0$ or $1$.
Now, we come to the proof of Theorem \ref{thm:visibility}.

\begin{proof}[Proof of Theorem \ref{thm:visibility}]
First, let us set some notation.
For every $x\neq y\in\mathbb{T}$, we denote by $\llbracket x,y\rrbracket$ the geodesic path between $x$ and $y$ in $\mathbb{T}$, and we let ${(x\leftrightarrow y)=(\Pi\langle\llbracket x,y\rrbracket\rangle=0)}$ be the event ``the vacant set $\mathcal{V}$ contains $\llbracket x,y\rrbracket$''.
For every $n\in\mathbb{N}^*$, we denote by $[\mathbb{T},\varnothing]_n=\{x\in\mathbb{T}:d(\varnothing,x)\leq n\}$ the set of vertices within graph distance $n$ from the root, and we let $\partial[\mathbb{T},\varnothing]_n=\{x\in\mathbb{T}:d(\varnothing,x)=n\}$.
Now, for every $n\in\mathbb{N}^*$, let $\mathcal{Z}_n=\{x\in\partial[\mathbb{T},\varnothing]_n:\varnothing\leftrightarrow x\}$, and let $Z_n=\#\mathcal{Z}_n$.
We claim that $(Z_n)_{n\in\mathbb{N}^*}$ is a branching process.
Indeed, for each $n\in\mathbb{N}^*$, let $\mathcal{F}_n$ be the $\sigma$-algebra generated by the restriction of $\Pi$ to the set of lines that hit $[\mathbb{T},\varnothing]_n$, and consider the identity
\[Z_{n+1}=\sum_{x\in\partial[\mathbb{T},\varnothing]_n}\mathbf{1}(\varnothing\leftrightarrow x)\cdot\sum_\text{$y$ child of $x$}\mathbf{1}(\Pi\langle y1,y2\rangle=0).\]
On the one hand, the events $((\varnothing\leftrightarrow x)\,;\,x\in[\mathbb{T},\varnothing]_n)$ are $\mathcal{F}_n$-measurable; on the other hand, the events $((\Pi\langle y1,y2\rangle)\,;\,y\in\partial[\mathbb{T},\varnothing]_{n+1})$ are independent, and independent of $\mathcal{F}_n$.
This shows that conditionally on $\mathcal{F}_n$, the random variable $Z_{n+1}$ is distributed as a sum of $Z_n$ independent binomial random variables with $2$ trials and success probability $\mathbb{P}\left(\Pi\langle y1,y2\rangle=0\right)=e^{-\alpha\cdot\mu\langle y1,y2\rangle}=e^{-\alpha/4}$.
Now, let us complete the proof of the proposition.
\begin{itemize}
\item For $\alpha\geq4\ln2$, we have $\mathbb{E}\left[\mathrm{Binomial}\left(2,e^{-\alpha/4}\right)\right]\leq1$.
By standard branching processes results, we get that almost surely, we have $Z_n=0$ for all sufficiently large $n$, which readily implies that almost surely, the vacant set $\mathcal{V}$ does not contain any ray.
By invariance, we deduce that for every $x\in\mathbb{T}$, the event ``the vacant set $\mathcal{V}$ contains an infinite geodesic path $(x=x_0,x_1,\ldots)$'' has probability $0$, and it follows that almost surely, the vacant set $\mathcal{V}$ does not contain any half-line.

\item For $\alpha<4\ln2$, we have $\mathbb{E}\left[\mathrm{Binomial}\left(2,e^{-\alpha/4}\right)\right]>1$.
Moreover, note that ${\mathbb{P}(Z_1>0)>0}$. 
By standard branching processes results, we get that with positive probability, we have $Z_n>0$ for all $n\in\mathbb{N}^*$.
Thus, by K\H{o}nig's lemma (see, e.g, \cite[Exercise 1.1]{lyonsperes}), with positive probability, say probability $\delta>0$, the vacant set $\mathcal{V}$ contains a ray.
It follows that there exists $i\in\{1,2,3\}$ such that the event $A_i$: ``the vacant set $\mathcal{V}$ contains a ray that passes through the vertex $i$'' has probability at least $\delta/3$; but note that by invariance, the probability of $A_i$ does not depend on $i$.
Therefore, by the Harris--FKG inequality for Poisson processes (see, e.g, \cite[Theorem 20.4]{lastpenrose}), since $A_1$ and $A_2$ are both decreasing events (adding lines to $\Pi$ inhibits them), we get
\[\mathbb{P}(A_1\cap A_2)\geq\mathbb{P}(A_1)\cdot\mathbb{P}(A_2)\geq\left(\frac{\delta}{3}\right)^2>0.\]
Since on the event $A_1\cap A_2$, the vacant set $\mathcal{V}$ contains a line that passes through the root, we deduce that the percolation event $A$: ``the vacant set $\mathcal{V}$ contains a line'' has positive probability.
Finally, observe that $A$ is invariant.
Since $\Pi$ is mixing, we must have $\mathbb{P}(A)=1$.
\end{itemize}
\end{proof}

\section{Driving to infinity with a Poisson process of roads}\label{sec:roads}

In this section, we prove Theorem \ref{thm:explosion}.
Following Aldous \cite{aldous} and Kendall \cite{kendall}, we let $\Pi$ be a Poisson process with intensity measure $\nu$ proportional to $\mu\otimes v^{-\beta}\mathrm{d}v$ on $\mathbb{L}\times\mathbb{R}_+^*$, where $\beta>1$ is a parameter, and $\mu$ is the invariant measure on $\mathbb{L}$ normalised by \eqref{eq:munorm}.
Viewing each atom $(\ell,v)$ of $\Pi$ as a road in $\mathbb{T}$, with $v$ the speed limit on the line $\ell$, we consider the random metric $T:\mathbb{T}\times\mathbb{T}\rightarrow\mathbb{R}_+$ induced by the driving distance with respect to the road network generated by $\Pi$.
Unlike in the Euclidean case \cite{kendall,kahn,moa}, there is no issue in defining this driving distance metric for all values of $\beta>1$ (see just below), and we consider its ``explosive'' properties.
More precisely, let us recall the statement of Theorem \ref{thm:explosion}.
\begin{itemize}
\item For $\beta<2$, we have $T(\varnothing,\partial\mathbb{T})<\infty$ almost surely; i.e, almost surely, there exists a ray $(\varnothing=x_0,x_1,\ldots)$ such that $\sum_{n\geq1}T(x_{n-1},x_n)<\infty$.
\item For $\beta>2$, we have $T(\varnothing,\partial\mathbb{T})=\infty$ almost surely.
\end{itemize}

Before proving Theorem \ref{thm:explosion}, we present the model in more detail.
Then, we consider the phase $\beta<2$ in Subsection \ref{subsec:explosion} (see Proposition \ref{prop:greedy}), and the phase $\beta>2$ in Subsection \ref{subsec:nonexplosive} (see Proposition \ref{prop:noexplosion}).

\paragraph{The Poisson process of roads $\Pi$.}

We let $\Pi$ be a Poisson process with intensity measure ${\nu=(\beta-1)\cdot\mu\otimes v^{-\beta}\mathrm{d}v}$ on $\mathbb{L}\times\mathbb{R}_+^*$, where $\beta>1$ is a parameter.
The normalising constant $(\beta-1)$ is here for convenience, so that
\[\nu(\langle x,y\rangle\times[v,\infty[)=2^{-d(x,y)}\cdot v_0^{-(\beta-1)}\quad\text{for all $x\neq y\in\mathbb{T}$ and $v_0\in\mathbb{R}_+^*$.}\]
In fact, this multiplicative constant does not affect the result of Theorem \ref{thm:explosion}, as multiplying $\nu$ by a constant factor does not change the probability of the explosion event $(T(\varnothing,\partial\mathbb{T})<\infty)$.
Indeed, if $\Pi_\alpha$ is a Poisson process with intensity $\alpha\cdot\nu$ on $\mathbb{L}\times\mathbb{R}_+^*$, where $\alpha>0$, then a change of variables shows that 
\[\Pi_\alpha\overset{\text{\tiny law}}{=}\left\{\left(\ell,\alpha^{1/(\beta-1)}\cdot v\right);\;(\ell,v)\in\Pi\right\}.\]
It follows that the metric $T_\alpha$ induced by $\Pi_\alpha$ has the same distribution as $\alpha^{-1/(\beta-1)}\cdot T$.
In particular, explosion occurs for $T_\alpha$ with the same probability as for $T  $.
As usual with well-behaved Poisson processes, we view $\Pi$ sometimes as a random subset of $\mathbb{L}\times\mathbb{R}_+^*$, and sometimes as a random atomic measure on $\mathbb{L}\times\mathbb{R}_+^*$, without making the distinction.
Note that there is no multiplicity ambiguity here, since almost surely, we have $\Pi(\mathbb{L}\times\{v\})\leq1$ for all $v\in\mathbb{R}_+^*$.
We recall that an atomic measure on $\mathbb{L}\times\mathbb{R}_+^*$ is a measure of the form $\pi=\sum_{k=1}^n\delta_{(\ell_k,v_k)}$, with $n\in\llbracket0,\infty\rrbracket$ and $(\ell_k,v_k)\in\mathbb{L}\times\mathbb{R}_+^*$ for every $k$.
We denote by $\mathbb{M}$ the space of atomic measures on $\mathbb{L}\times\mathbb{R}_+^*$, equipped with the $\sigma$-algebra generated by the maps
\[\begin{matrix}
\mathbb{M}&\longrightarrow&\llbracket0,\infty\rrbracket\\
\pi&\longmapsto&\pi(B),
\end{matrix}\]
for $B$ Borel subset of $\mathbb{L}\times\mathbb{R}_+^*$.
By construction, the Poisson process $\Pi$ has the following invariance property: for any graph automorphism $\phi:\mathbb{T}\rightarrow\mathbb{T}$, we have
\[\phi_*\Pi:=\left\{(\phi(\ell),v)\,;\,(\ell,v)\in\Pi\right\}\overset{\text{\tiny law}}{=}\Pi.\]
Moreover, the Poisson process $\Pi$ is mixing: if $(\psi_n)_{n\in\mathbb{N}}$ is a sequence of graph automorphisms of $\mathbb{T}$ such that $d(\varnothing,\psi_n(\varnothing))\rightarrow\infty$ as $n\to\infty$, then for every bounded measurable functions $f$ and $g$ from $\mathbb{M}$ to $\mathbb{R}$, we have
\[\mathbb{E}[f(\Pi)\cdot g(\psi_n^*\Pi)]\longrightarrow\mathbb{E}[f(\Pi)]\cdot\mathbb{E}[g(\Pi)]\quad\text{as $n\to\infty$,}\]
where $\psi_n^*\Pi=\{(\psi_n(\ell),v)\,;\,(\ell,v)\in\Pi\}$.
In particular, any invariant event has probability $0$ or $1$.

\paragraph{Construction of the metric $T$.}

Now, let us construct the driving distance metric $T$ induced by $\Pi$ more precisely.
For every $x,y\in\mathbb{T}$, we let 
\[T(x,y)=V_{e_1}^{-1}+\ldots+V_{e_n}^{-1},\]
where $e_1,\ldots,e_n$ denote the edges on the geodesic path between $x$ and $y$ in $\mathbb{T}$, and where $V_e$ denotes the speed of the fastest road of $\Pi$ that passes through $e$, for each edge $e$ of $\mathbb{T}$.
More generally, for every $x\neq y\in\mathbb{T}$, we denote by $V_{x,y}$ the speed of the fastest road of $\Pi$ that passes through both points $x$ and $y$.
We have
\[\mathbb{P}(V_{x,y}<v)=\mathbb{P}(\Pi(\langle x,y\rangle\times[v,\infty[)=0)=\exp\left[-2^{-d(x,y)}\cdot v^{-(\beta-1)}\right]\quad\text{for all $v\in\mathbb{R}_+^*$.}\]
In particular, the random variables $(1/V_e,\,\text{$e$ edge of $\mathbb{T}$})$ are well-defined, with values in $\mathbb{R}_+^*$.
It follows that almost surely, the function $T:\mathbb{T}\times\mathbb{T}\rightarrow\mathbb{R}_+$ is a metric on $\mathbb{T}$.
Equivalently, this driving distance metric is the first passage percolation distance function associated with the passage times $\left(1/V_e,\,\text{$e$ edge of $\mathbb{T}$}\right)$.
By the Harris--FKG inequality for Poisson processes (see, e.g, \cite[Theorem 20.4]{lastpenrose}), these passage times are positively associated, as nondecreasing functions of $\Pi$.
For future reference, note that we have
\begin{equation}\label{eq:tau}
T(x,y)=\tau(\Pi;x,y)\quad\text{for all $x,y\in\mathbb{T}$}
\end{equation}
for some measurable function $\tau:\mathbb{M}\times(\mathbb{T}\times\mathbb{T})\rightarrow\mathbb{R}_+$.

\subsection{Explosion and the greedy process}\label{subsec:explosion}

In this subsection, we prove that for $\beta<2$, we have $T(\varnothing,\partial\mathbb{T})<\infty$ almost surely.
Following Pemantle and Peres (see \cite[proof of Theorem 3]{pemantleperes}), we consider the greedy process $(X_n)_{n\in\mathbb{N}}$ on $\mathbb{T}$ which starts at the root and follows the fastest road at each step.
More precisely, let $X_0=\varnothing$, and for each $n\in\mathbb{N}$, let $X_{n+1}$ be the child $X_ni$ of $X_n$ with minimal label $i$ (to break ties) that minimises the passage time $T(X_n,X_ni)=V_{X_n,X_ni}^{-1}$.
The following proposition tells us that for $\beta<2$, this process reaches the boundary of $\mathbb{T}$ in finite time a.s.

\begin{prop}\label{prop:greedy}
The greedy process undergoes a phase transition at $\beta=2$.
\begin{itemize}
\item For $\beta<2$, we have $\sum_{n\geq1}T(X_{n-1},X_n)<\infty$ almost surely.
\item For $\beta\geq2$, we have $\sum_{n\geq1}T(X_{n-1},X_n)=\infty$ almost surely.
\end{itemize}
\end{prop}
\begin{proof}
By the definition of $(X_n)_{n\in\mathbb{N}}$, we have $T(X_n,X_{n+1})=T(X_{n-1},X_n)\wedge V_{X_n1,X_n2}^{-1}$ for each $n\in\mathbb{N}^*$.
Therefore, we have
\[T(X_{n-1},X_n)=T(X_0,X_1)\wedge W_2\wedge\ldots\wedge W_n\quad\text{for all $n\in\mathbb{N}^*$,}\]
where $W_n=V_{X_{n-1}1,X_{n-1}2}^{-1}$ for all $n\in\mathbb{N}^*$.
In particular, the sum $\sum_{n\geq1}T(X_{n-1},X_n)$ has the same nature as ${\sum_{n\geq1}W_1\wedge\ldots\wedge W_n}$.
To conclude the proof, let us show that almost surely, the last sum has the same nature as $\sum_{n\geq1}n^{-1/(\beta-1)}$.
First, notice that the random variables $(W_n)_{n\in\mathbb{N}^*}$ are independent and identically distributed.
Indeed, for each $n\in\mathbb{N}^*$, let $\mathcal{F}_n$ be the $\sigma$-algebra generated by the restriction of $\Pi$ to the set of roads that hit $[\mathbb{T},\varnothing]_{n-1}=\{x\in\mathbb{T}:d(\varnothing,x)\leq n-1\}$.
On the one hand, the random variables $X_n$ and $W_n$ are $\mathcal{F}_n$-measurable.
On the other hand, the random variable $W_{n+1}$ is independent of $\mathcal{F}_n$, and distributed as $V_{1,2}^{-1}$.
Indeed, for every $t\in\mathbb{R}_+^*$, we have
\[\begin{split}
\mathbb{P}(W_{n+1}>t\,|\,\mathcal{F}_n)&=\sum_{x\in\partial[\mathbb{T},\varnothing]_n}\mathbb{P}\left(X_n=x\,;\,V_{x1,x2}^{-1}>t\,\middle|\,\mathcal{F}_n\right)\\
&=\sum_{x\in\partial[\mathbb{T},\varnothing]_n}\mathbf{1}(X_n=x)\cdot\mathbb{P}\left(V_{x1,x2}^{-1}>t\right)=\exp\left[-1/4\cdot t^{\beta-1}\right].
\end{split}\]
Next, for every $n\in\mathbb{N}^*$, let $Y_n=n^{1/(\beta-1)}\cdot W_1\wedge\ldots\wedge W_n$.
We have
\[\mathbb{P}(Y_n>t)=\mathbb{P}\left(W_1>n^{-1/(\beta-1)}\cdot t\right)^n=\exp\left[-1/4\cdot t^{\beta-1}\right]\quad\text{for all $t\in\mathbb{R}_+^*$;}\]
hence, the $(Y_n)_{n\in\mathbb{N}^*}$ are identically distributed random variables, with values in $\mathbb{R}_+^*$.
Moreover, we have $\mathbb{E}[Y_1]<\infty$.
The result claimed above now follows from the Jeulin lemma (see \cite[Theorem 3.1 and Proposition 3.2]{jeulin}): almost surely, the sum
\[\sum_{n\geq1}W_1\wedge\ldots\wedge W_n=\sum_{n\geq1}\frac{Y_n}{n^{1/(\beta-1)}}\]
has the same nature as $\sum_{n\geq1}n^{-1/(\beta-1)}$.
\end{proof}

\subsection{Non-explosion and the bounded driving distance probability}\label{subsec:nonexplosive}

In this subsection, we consider the phase $\beta>2$.
We first prove that there is no explosion in this phase.
Then, we study the so-called bounded driving distance probability; namely, the probability $\mathbb{P}(T(\varnothing,1_n)\leq t)$ that the driving distance between two vertices at distance $n$ in $\mathbb{T}$ is at most $t$, for fixed $t>0$ and as $n\to\infty$.

\subsubsection{Non-explosion}\label{subsubsec:noexplosion}

Now, we prove that for $\beta>2$, we have $T(\varnothing,\partial\mathbb{T})=\infty$ almost surely.
First, consider the following easy lemma.

\begin{lem}
If there exists $t>0$ such that the driving distance ball $\{x\in\mathbb{T}:T(\varnothing,x)\leq t\}$ is finite a.s, then we have $T(\varnothing,\partial\mathbb{T})=\infty$ almost surely.
\end{lem}
\begin{proof}
Let $t>0$ be such that $\#\{x\in\mathbb{T}:T(\varnothing,x)\leq t\}<\infty$ almost surely.
For every $x\in\mathbb{T}$, let $A_x$ be the event: ``there exists an infinite geodesic path $(x=x_0,x_1,\ldots)$ in $\mathbb{T}$ such that ${\sum_{n\geq1}T(x_{n-1},x_n)\leq t}$''.
Notice that the explosion event $(T(\varnothing,\partial\mathbb{T})<\infty)$ is contained in $\bigcup_{x\in\mathbb{T}}A_x$.
On the other hand, we have $\mathbb{P}(A_x)=\mathbb{P}(A_\varnothing)=0$ for all $x\in\mathbb{T}$, where the first equality holds by invariance, and the second by assumption.
Finally, we obtain $\mathbb{P}(T(\varnothing,\partial\mathbb{T})<\infty)=0$.
\end{proof}

By the previous lemma, it suffices to prove that for $t$ small enough, the driving distance ball $\{x\in\mathbb{T}:T(\varnothing,x)\leq t\}$ is finite almost surely.
This is the crux of the proof.

\begin{prop}\label{prop:noexplosion}
For $\beta>2$, there exists $t>0$ such that
\[\mathbb{E}[\#\{x\in\mathbb{T}:T(\varnothing,x)\leq t\}]<\infty.\]
\end{prop}
\begin{proof}
We keep denoting by $[\mathbb{T},\varnothing]_n$ the set of vertices within graph distance $n$ from the root.
More generally, for $S\subset\mathbb{T}$ and for $x\in\mathbb{T}$, we let $[S,x]_n=\{y\in S:d(x,y)\leq n\}$.
Now, let $t\in{]0,1/9]}$ be a parameter to be adjusted later, and let
\[\varphi_n^*=\mathbb{E}[\#\{x\in[\mathbb{T},\varnothing]_n:T(\varnothing,x)\leq t\}]\quad\text{for all $n\in\mathbb{N}$.}\]
To start working on the $\varphi_n^*$ terms, we would like to integrate on the speed of the fastest road of $\Pi$ that passes through the root.
A rigorous way of doing that is to use the Slivnyak--Mecke theorem, that we recall now.
For $(\ell,v)\in\mathbb{L}\times\mathbb{R}_+^*$ and $\pi\in\mathbb{M}$, we denote by $(\ell,v)\oplus\pi$ (resp. $\pi\ominus(\ell,v)$) the atomic measure obtained from $\pi$ by adding (resp. removing) the atom $(\ell,v)$.
The Slivnyak--Mecke theorem (see, e.g, \cite[Theorem 4.1]{lastpenrose}) states that for every measurable function $f:\left(\mathbb{L}\times\mathbb{R}_+^*\right)\times\mathbb{M}\rightarrow\mathbb{R}_+$, we have
\[\mathbb{E}\left[\sum_{(\ell,v)\in\Pi}f(\ell,v;\Pi)\right]=\int_{\mathbb{L}\times\mathbb{R}_+^*}\mathbb{E}[f(\ell,v;(\ell,v)\oplus\Pi)]\mathrm{d}\nu(\ell,v).\]
Equivalently, for every measurable function $g:\left(\mathbb{L}\times\mathbb{R}_+^*\right)\times\mathbb{M}\rightarrow\mathbb{R}_+$, we have
\[\mathbb{E}\left[\sum_{(\ell,v)\in\Pi}g(\ell,v;\Pi\ominus(\ell,v))\right]=\int_{\mathbb{L}\times\mathbb{R}_+^*}\mathbb{E}[g(\ell,v;\Pi)]\mathrm{d}\nu(\ell,v).\]
Now, consider the following lemma.
For every $x\in\mathbb{T}$, we denote by $(L_x,V_x)$ the fastest road of $\Pi$ that passes through $x$.

\begin{lem}\label{lem:mecke'}
Let $x\in\mathbb{T}$.
For every measurable function $F:\left(\mathbb{L}\times\mathbb{R}_+^*\right)\times\mathbb{M}\rightarrow\mathbb{R}_+$, we have
\[\mathbb{E}[F(L_x,V_x;\Pi)]=\int_{\langle x\rangle\times\mathbb{R}_+^*}\mathbb{E}[F(\ell,v;(\ell,v)\oplus\Pi)\,;\,V_x<v]\mathrm{d}\nu(\ell,v).\]
Equivalently, for every measurable function $G:\left(\mathbb{L}\times\mathbb{R}_+^*\right)\times\mathbb{M}\rightarrow\mathbb{R}_+$, we have
\[\mathbb{E}[G(L_x,V_x;\Pi\ominus(L_x,V_x))]=\int_{\langle x\rangle\times\mathbb{R}_+^*}\mathbb{E}[G(\ell,v;\Pi)\,;\,V_x<v]\mathrm{d}\nu(\ell,v).\]
\end{lem}
\begin{proof}[Proof of the lemma]
Let us prove the first identity only, the second one is an immediate consequence.
Let $F:\left(\mathbb{L}\times\mathbb{R}_+^*\right)\times\mathbb{M}\rightarrow\mathbb{R}_+$ be a measurable function.
We have
\[F(L_x,V_x;\Pi)=\sum_{(\ell,v)\in\Pi}g(\ell,v;\Pi\ominus(\ell,v)),\]
where $g:\left(\mathbb{L}\times\mathbb{R}_+^*\right)\times\mathbb{M}\rightarrow\mathbb{R}_+$ is the measurable function defined by
\[g(\ell,v;\pi)=\begin{cases}
F(\ell,v;(\ell,v)\oplus\pi)&\text{if $\ell\in\langle x\rangle$ and $\pi(\langle x\rangle\times[v,\infty[)=0$}\\
0&\text{otherwise}
\end{cases}\]
for all $(\ell,v;\pi)\in\left(\mathbb{L}\times\mathbb{R}_+^*\right)\times\mathbb{M}$.
By the Slivnyak--Mecke theorem, it follows that
\[\mathbb{E}[F(L_x,V_x;\Pi)]=\int_{\mathbb{L}\times\mathbb{R}_+^*}\mathbb{E}[g(\ell,v;\Pi)]\mathrm{d}\nu(\ell,v)=\int_{\langle x\rangle\times\mathbb{R}_+^*}\mathbb{E}[F(\ell,v;(\ell,v)\oplus\Pi)\,;\,V_x<v]\mathrm{d}\nu(\ell,v).\]
\end{proof}

Back to the proof of the proposition, we apply Lemma \ref{lem:mecke'} with $x=\varnothing$ and
\[F(\ell,v;\pi)=\#\{y\in[\mathbb{T},\varnothing]_n:\tau(\pi;\varnothing,y)\leq t\}\quad\text{for all $(\ell,v;\pi)\in\left(\mathbb{L}\times\mathbb{R}_+^*\right)$,}\]
where $\tau$ is the measurable function of \eqref{eq:tau}.
We obtain
\[\varphi_n^*=\int_{\langle\varnothing\rangle\times\mathbb{R}_+^*}\mathbb{E}[\#\{y\in[\mathbb{T},\varnothing]_n:\tau((\ell,v)\oplus\Pi;\varnothing,y)\leq t\}\,;\,V_\varnothing<v]\mathrm{d}\nu(\ell,v).\]
Now, let us slightly abuse notation, and denote the integrand by $\varphi_n(v)$, as it does not depend on $\ell$.
Indeed, by invariance, we have
\begin{eqnarray*}
\varphi_n(v)&=&\mathbb{E}[\#\{y\in[\mathbb{T},\varnothing]_n:\tau((\ell,v)\oplus\Pi;\varnothing,y)\leq t\}\,;\,V_\varnothing<v]\\
&=&\mathbb{E}\left[\#\left\{y\in\left[\mathbb{T},x'\right]_n:\tau\left(\left(\ell',v\right)\oplus\Pi;x',y\right)\leq t\right\}\,;\,V_{x'}<v\right]
\end{eqnarray*}
for every $\ell'\in\mathbb{L}$, and for every $x'\in\ell'$.
Note that
\[\varphi_n(v)=\mathbb{P}(V_\varnothing<v)=\exp\left[-3/4\cdot v^{-(\beta-1)}\right]\quad\text{for all $v\in\left]0,\frac{1}{t}\right[$.}\]
Furthermore, we claim that for every $n\in\mathbb{N}$, we have
\begin{equation}\label{eq:recurrence}
\varphi_{n+1}(v)\leq1+\varphi_n^*+2\cdot t\cdot v\cdot\left(\varphi_n^*+\int_{\langle\varnothing\rangle\times{]v,\infty[}}\varphi_n\left(v'\right)\mathrm{d}\nu\left(\ell',v'\right)\right)\quad\text{for all $v\in\left[\frac{1}{t},\infty\right[$.}
\end{equation}
First, assuming that this holds, let us complete the proof of the proposition.
Since $\beta>2$, we can adjust the parameter $t$ so that $\int_{\langle\varnothing\rangle\times{]1/t,\infty[}}v\mathrm{d}\nu(\ell,v)\leq1$.
Now, let us prove by induction that for every $n\in\mathbb{N}$, we have $\varphi_n(v)\leq v$ for all $v\in[1/t,\infty[$.
This is obviously true for $n=0$, since $\varphi_0(v)=\mathbb{P}(V_\varnothing<v)$ for all $v\in\mathbb{R}_+^*$.
Next, let $n\in\mathbb{N}$, assume that $\varphi_n(v)\leq v$ for all $v\in[1/t,\infty[$, and let us prove that $\varphi_{n+1}(v)\leq v$ for all $v\in[1/t,\infty[$.
First, note that
\[\varphi_n^*=\int_{\langle\varnothing\rangle\times\mathbb{R}_+^*}\varphi_n(v)\mathrm{d}\nu(\ell,v)\leq\int_{\langle\varnothing\rangle\times\mathbb{R}_+^*}\mathbb{P}(V_\varnothing<v)\mathrm{d}\nu(\ell,v)+\int_{\langle\varnothing\rangle\times{]1/t,\infty[}}v\mathrm{d}\nu(\ell,v)\leq2,\]
where we use Lemma \ref{lem:mecke'} with $x=\varnothing$ and $F\equiv1$ to check that $\int_{\langle\varnothing\rangle\times\mathbb{R}_+^*}\mathbb{P}(V_\varnothing<v)\mathrm{d}\nu(\ell,v)=1$.
Now, it follows from \eqref{eq:recurrence} that for every $v\in[1/t,\infty[$, we have
\[\varphi_{n+1}(v)\leq1+2+2\cdot t\cdot v\cdot\left(2+\int_{\langle\varnothing\rangle\times{]1/t,\infty[}}v'\mathrm{d}\nu\left(\ell',v'\right)\right)\leq3+6\cdot t\cdot v\leq v,\]
where the last inequality holds since $t\leq1/9$.
By induction, this proves that for every $n\in\mathbb{N}$, we have $\varphi_n(v)\leq v$ for all $v\in[1/t,\infty[$.
In turn, this implies that for every $n\in\mathbb{N}$, we have
\[\varphi_n^*=\int_{\langle\varnothing\rangle\times\mathbb{R}_+^*}\varphi_n(v)\mathrm{d}\nu(\ell,v)\leq\int_{\langle\varnothing\rangle\times\mathbb{R}_+^*}\mathbb{P}(V_\varnothing<v)\mathrm{d}\nu(\ell,v)+\int_{\langle\varnothing\rangle\times{]1/t,\infty[}}v\mathrm{d}\nu(\ell,v)\leq2.\]
Letting $n\to\infty$ in the definition of $\varphi_n^*$, we obtain
\[\mathbb{E}[\#\{x\in\mathbb{T}:T(\varnothing,x)\leq t\}]\leq 2\]
by the monotone convergence theorem, which gives the result of the proposition.

To complete the proof, it remains to establish \eqref{eq:recurrence}.
Let $n\in\mathbb{N}$, and fix $(\ell,v)\in\langle\varnothing\rangle\times[1/t,\infty[$.
We denote by $\rho_1$ and $\rho_2$ the two neighbours of $\varnothing$ that are on $\ell$, and by $\rho_3$ the neighbour of $\varnothing$ that is not on $\ell$.
For each $i\in\{1,2,3\}$, we let $S_i=\{x\in\mathbb{T}:x\succeq\rho_i\}$.
See Figure \ref{fig:noexplosion1} for an illustration.

\begin{figure}[ht]
\centering
\includegraphics[width=0.5\linewidth]{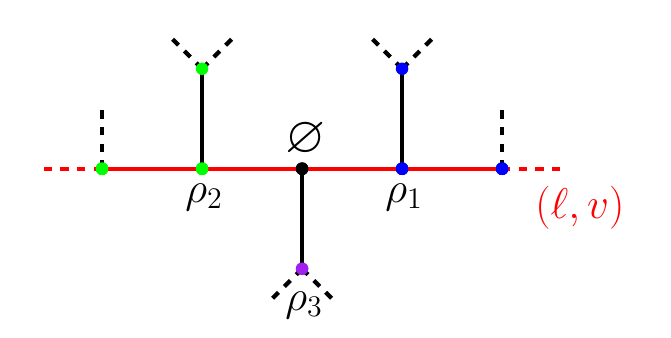}
\caption{The vertices of $S_1$ are in blue, the vertices of $S_2$ are in green, and the vertices of $S_3$ are in purple.}\label{fig:noexplosion1}
\end{figure}

We have
\begin{eqnarray*}
\lefteqn{\#\{x\in[\mathbb{T},\varnothing]_{n+1}:\tau((\ell,v)\oplus\Pi;\varnothing,x)\leq t\}}\\
&\leq&1+\sum_{i=1}^2\#\{x\in[S_i,\rho_i]_n:\tau((\ell,v)\oplus\Pi;\varnothing,x)\leq t\}+\#\left\{x\in[S_3,\rho_3]_n:\tau(\Pi;\rho_3,x)\leq t\right\}\\
&\leq&1+\sum_{i=1}^2\#\{x\in[S_i,\rho_i]_n:\tau((\ell,v)\oplus\Pi;\varnothing,x)\leq t\}+\#\left\{x\in[\mathbb{T},\rho_3]_n:\tau(\Pi;\rho_3,x)\leq t\right\}.
\end{eqnarray*}
It follows that
\begin{eqnarray*}
\lefteqn{\mathbb{E}[\#\{x\in[\mathbb{T},\varnothing]_{n+1}:\tau((\ell,v)\oplus\Pi;\varnothing,x)\leq t\}\,;\,V_\varnothing<v]}\\
&\leq&1+\sum_{i=1}^2\mathbb{E}[\#\{x\in[S_i,\rho_i]_n:\tau((\ell,v)\oplus\Pi;\varnothing,x)\leq t\}\,;\,V_\varnothing<v]\\
&&+\mathbb{E}[\#\{x\in[\mathbb{T},\rho_3]_n:\tau(\Pi;\rho_3,x)\leq t\}],
\end{eqnarray*}
where we recognise $\mathbb{E}[\#\{x\in[\mathbb{T},\rho_3]_n:\tau(\Pi;\rho_3,x)\leq t\}]=\varphi_n^*$, and it remains to handle the sum of two terms.
Since the two terms are equal by invariance, let us focus on the first of them.
To prove \eqref{eq:recurrence}, it suffices to show that
\[\mathbb{E}[\#\{x\in[S_1,\rho_1]_n:\tau((\ell,v)\oplus\Pi;\varnothing,x)\leq t\}\,;\,V_\varnothing<v]\leq t\cdot v\cdot\left(\varphi_n^*+\int_{\langle\varnothing\rangle\times{]v,\infty[}}\varphi_n\left(v'\right)\mathrm{d}\nu\left(\ell',v'\right)\right).\]
Let us denote by $x_1,x_2,\ldots$ the vertices of $S_1$ that are on $\ell$, with $d(\varnothing,x_i)=i$ for all $i\in\mathbb{N}^*$.
In particular, we have $x_1=\rho_1$.
For each $i\in\mathbb{N}^*$, we denote by $x'_i$ the neighbour of $x_i$ that is not on $\ell$, and we let $S'_i=\{x_i\}\cup\left\{x\in\mathbb{T}:x\succeq x'_i\right\}$.
Moreover, we let $S''_i=\{x\in\mathbb{T}:x\succeq x_i\}$.
See Figure \ref{fig:noexplosion2} for an illustration.

\begin{figure}[ht]
\centering
\includegraphics[width=0.6\linewidth]{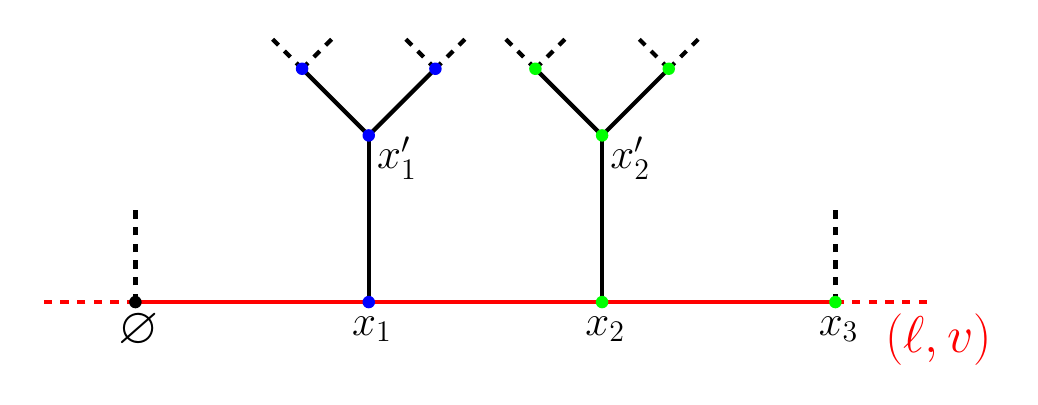}
\caption{The vertices of $S'_1$ are in blue, and the vertices of $S''_2$ are in green.}\label{fig:noexplosion2}
\end{figure}

Let $k=\lfloor t\cdot v\rfloor$ be the largest integer $i\in\mathbb{N}^*$ such that $d(\varnothing,x_i)\leq t\cdot v$.
We claim that
\begin{eqnarray*}
\lefteqn{\mathbb{E}[\#\{x\in[S_1,\rho_1]_n:\tau((\ell,v)\oplus\Pi;\varnothing,x)\leq t\}\,;\,V_\varnothing<v]}\\
&\leq&\sum_{i=1}^k\mathbb{E}\left[\#\{x\in[\mathbb{T},x_i]_n:\tau(\Pi;x_i,x)\leq t\}\right]\\
&&+\sum_{j=1}^k\mathbb{E}\left[\#\left\{x\in\left[\mathbb{T},x_j\right]_n:\tau\left(\left(\ell,V_{x_j}\right)\oplus\Pi\ominus\left(L_{x_j},V_{x_j}\right);x_j,x\right)\leq t\right\}\,;\,V_{x_j}>v\right].
\end{eqnarray*}
Indeed, on the event $(V_\varnothing<v)$, we have the following alternative.
\begin{itemize}
\item If no road of $\Pi$ with speed more than $v$ hits $\llbracket x_1,x_k\rrbracket$, then $(\ell,v)$ is the fastest road of $(\ell,v)\oplus\Pi$ that passes through each edge on the geodesic path between $\varnothing$ and $x_{k+1}$ in $\mathbb{T}$.
Therefore, we have $\tau((\ell,v)\oplus\Pi;\varnothing,x_{k+1})=(k+1)/v>t$, and it follows that
\[\begin{split}
\#\{x\in[S_1,\rho_1]_n:\tau((\ell,v)\oplus\Pi;\varnothing,x)\leq t\}&\leq\sum_{i=1}^k\#\left\{x\in\left[S'_i,x_i\right]_n:\tau(\Pi;x_i,x)\leq t\right\}\\
&\leq\sum_{i=1}^k\#\{x\in[\mathbb{T},x_i]_n:\tau(\Pi;x_i,x)\leq t\}.
\end{split}\]

\item Otherwise, let $j$ be the smallest integer $i\in\llbracket1,k\rrbracket$ such that $V_{x_i}>v$.
We have
\begin{eqnarray*}
\lefteqn{\#\{x\in[S_1,\rho_1]_n:\tau((\ell,v)\oplus\Pi;\varnothing,x)\leq t\}}\\
&\leq&\sum_{i=1}^{j-1}\#\left\{x\in\left[S'_i,x_i\right]_n:\tau(\Pi;x_i,x)\leq t\right\}+\#\left\{x\in\left[S''_j,x_j\right]_n:\tau((\ell,v)\oplus\Pi;x_j,x)\leq t\right\}\\
&\leq&\sum_{i=1}^{j-1}\#\{x\in[\mathbb{T},x_i]_n:\tau(\Pi;x_i,x)\leq t\}+\#\left\{x\in\left[S''_j,x_j\right]_n:\tau((\ell,v)\oplus\Pi;x_j,x)\leq t\right\}.
\end{eqnarray*}
To bound the last term, consider the fastest road $\left(L_{x_j},V_{x_j}\right)$ of $\Pi$ that passes through $x_j$.
Let $l$ be the largest integer $i\in\llbracket j+1,\infty\llbracket$ such that $L_{x_j}$ passes through $x_i$.
See Figure \ref{fig:noexplosion3} for an illustration.

\begin{figure}[ht]
\centering
\includegraphics[width=0.55\linewidth]{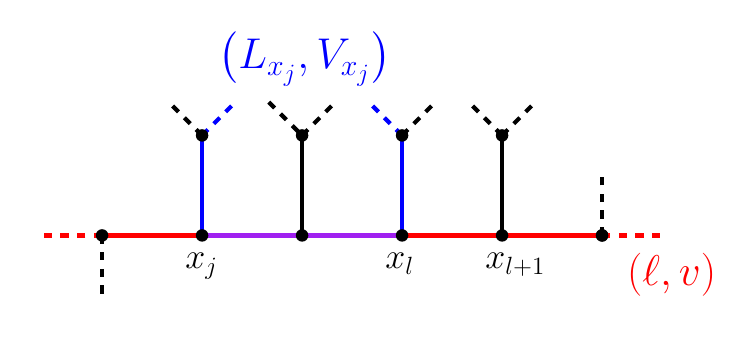}
\caption{The intersection of $\ell$ and $L_{x_j}$, which corresponds to the segment $\left\llbracket x_j,x_l\right\rrbracket$, is in purple.}\label{fig:noexplosion3}
\end{figure}

We decompose
\begin{eqnarray*}
\lefteqn{\#\left\{x\in\left[S''_j,x_j\right]_n:\tau((\ell,v)\oplus\Pi;x_j,x)\leq t\right\}}\\
&\leq&\#\left\{x\in\left[S'_j\cup\ldots\cup S'_l,x_j\right]_n:\tau(\Pi;x_j,x)\leq t\right\}\\
&&+\#\left\{x\in\left[S''_{l+1},x_j\right]_n:\tau((\ell,v)\oplus\Pi;x_j,x)\leq t\right\}\\
&\leq&\#\{x\in[\mathbb{T},x_j]_n:\tau(\Pi;x_j,x)\leq t\}\\
&&+\#\left\{x\in\left[S''_{l+1},x_j\right]_n:\tau((\ell,v)\oplus\Pi;x_j,x)\leq t\right\},
\end{eqnarray*}
where the first inequality holds since $\left(L_{x_j},V_{x_j}\right)$ is the fastest road of $(\ell,v)\oplus\Pi$ that passes through each edge on the geodesic path between $x_j$ and $x_l$.
For the last term, since $V_{x_j}>v$, we have
\begin{eqnarray*}
\lefteqn{\#\left\{x\in\left[S''_{l+1},x_j\right]_n:\tau((\ell,v)\oplus\Pi;x_j,x)\leq t\right\}}\\
&\leq&\#\left\{x\in\left[S''_{l+1},x_j\right]_n:\tau\left(\left(\ell,V_{x_j}\right)\oplus\Pi;x_j,x\right)\leq t\right\}\\
&=&\#\left\{x\in\left[S''_{l+1},x_j\right]_n:\tau\left(\left(\ell,V_{x_j}\right)\oplus\Pi\ominus\left(L_{x_j},V_{x_j}\right);x_j,x\right)\leq t\right\}\\
&\leq&\#\left\{x\in[\mathbb{T},x_j]_n:\tau\left(\left(\ell,V_{x_j}\right)\oplus\Pi\ominus\left(L_{x_j},V_{x_j}\right);x_j,x\right)\leq t\right\}.
\end{eqnarray*}
Note that we can even add the indicator $\mathbf{1}\left(V_{x_j}>v\right)$ in the right hand side.
Altogether, we obtain
\begin{eqnarray*}
\lefteqn{\#\{x\in[S_1,\rho_1]_n:\tau((\ell,v)\oplus\Pi;\varnothing,x)\leq t\}}\\
&\leq&\sum_{i=1}^{j-1}\#\{x\in[\mathbb{T},x_i]_n:\tau(\Pi;x_i,x)\leq t\}+\#\{x\in[\mathbb{T},x_j]_n:\tau(\Pi;x_j,x)\leq t\}\\
&&+\#\left\{x\in[\mathbb{T},x_j]_n:\tau\left(\left(\ell,V_{x_j}\right)\oplus\Pi\ominus\left(L_{x_j},V_{x_j}\right);x_j,x\right)\leq t\right\}\cdot\mathbf{1}\left(V_{x_j}>v\right).
\end{eqnarray*}
\end{itemize}
In any case, we get
\begin{eqnarray*}
\lefteqn{\#\{x\in[S_1,\rho_1]_n:\tau((\ell,v)\oplus\Pi;\varnothing,x)\leq t\}}\\
&\leq&\sum_{i=1}^k\#\{x\in[\mathbb{T},x_i]_n:\tau(\Pi;x_i,x)\leq t\}\\
&&+\sum_{j=1}^k\#\left\{x\in\left[\mathbb{T},x_j\right]_n:\tau\left(\left(\ell,V_{x_j}\right)\oplus\Pi\ominus\left(L_{x_j},V_{x_j}\right);x_j,x\right)\leq t\right\}\cdot\mathbf{1}\left(V_{x_j}>v\right).
\end{eqnarray*}
Recall that this holds on the event $(V_\varnothing<v)$.
The inequality claimed above follows by taking expectations:
\begin{eqnarray*}
\lefteqn{\mathbb{E}[\#\{x\in[S_1,\rho_1]_n:\tau((\ell,v)\oplus\Pi;\varnothing,x)\leq t\}\,;\,V_\varnothing<v]}\\
&\leq&\sum_{i=1}^k\mathbb{E}\left[\#\{x\in[\mathbb{T},x_i]_n:\tau(\Pi;x_i,x)\leq t\}\right]\\
&&+\sum_{j=1}^k\mathbb{E}\left[\#\left\{x\in\left[\mathbb{T},x_j\right]_n:\tau\left(\left(\ell,V_{x_j}\right)\oplus\Pi\ominus\left(L_{x_j},V_{x_j}\right);x_j,x\right)\leq t\right\}\,;\,V_{x_j}>v\right].
\end{eqnarray*}
Now, for each summand in the first term, we recognise $\mathbb{E}\left[\#\{x\in[\mathbb{T},x_i]_n:\tau(\Pi;x_i,x)\leq t\}\right]=\varphi_n^*$.
Next, for each summand in the second term, we use Lemma \ref{lem:mecke'} with $x=x_j$ and
\[G\left(\ell',v';\pi\right)=\begin{cases}
\#\left\{y\in\left[\mathbb{T},x_j\right]_n:\tau\left(\left(\ell,v'\right)\oplus\pi;x_j,y\right)\leq t\right\}&\text{if $v'>v$}\\
0&\text{otherwise}
\end{cases}\]
for all $\left(\ell',v';\pi\right)\in\left(\mathbb{L}\times\mathbb{R}_+^*\right)\times\mathbb{M}$.
We get
\begin{eqnarray*}
\lefteqn{\mathbb{E}\left[\#\left\{x\in\left[\mathbb{T},x_j\right]_n:\tau\left(\left(\ell,V_{x_j}\right)\oplus\Pi\ominus\left(L_{x_j},V_{x_j}\right);x_j,x\right)\leq t\right\}\,;\,V_{x_j}>v\right]}\\
&=&\int_{\langle x_j\rangle\times{]v,\infty[}}\mathbb{E}\left[\#\left\{x\in\left[\mathbb{T},x_j\right]_n:\tau\left(\left(\ell,v'\right)\oplus\Pi;x_j,x\right)\leq t\right\}\,;\,V_{x_j}<v'\right]\mathrm{d}\nu\left(\ell',v'\right),
\end{eqnarray*}
and we recognise $\varphi_n(v')$ as the integrand.
Altogether, we obtain
\begin{eqnarray*}
\lefteqn{\mathbb{E}[\#\{x\in[S_1,\rho_1]_n:\tau((\ell,v)\oplus\Pi;\varnothing,x)\leq t\}\,;\,V_\varnothing<v]}\\
&\leq&\sum_{i=1}^k\varphi_n^*+\sum_{j=1}^k\int_{\langle x_j\rangle\times{]v,\infty[}}\varphi_n\left(v'\right)\mathrm{d}\nu\left(\ell',v'\right)\\
&=&k\cdot\left(\varphi_n^*+\int_{\langle\varnothing\rangle\times{]v,\infty[}}\varphi_n\left(v'\right)\mathrm{d}\nu\left(\ell',v'\right)\right)\\
&\leq&t\cdot v\cdot\left(\varphi_n^*+\int_{\langle\varnothing\rangle\times{]v,\infty[}}\varphi_n\left(v'\right)\mathrm{d}\nu\left(\ell',v'\right)\right).
\end{eqnarray*}
This completes the proof of \eqref{eq:recurrence}, and concludes the proof of the proposition.
\end{proof}

\subsubsection{The bounded driving distance probability}\label{subsubsec:bddp}

In this paragraph, we study the so-called bounded driving distance probability; namely, the probability $\mathbb{P}(T(\varnothing,1_n)\leq t)$ that the driving distance between two points at distance $n$ in $\mathbb{T}$ is at most $t$, for fixed $t>0$ and as $n\to\infty$.
Note that we have the obvious lower bound
\[\mathbb{P}(T(\varnothing,1_n)\leq t)=\mathbb{P}\left(V_{\varnothing,1_n}\geq\frac{n}{t}\right)=1-\exp\left[-2^{-n}\cdot\left(\frac{t}{n}\right)^{\beta-1}\right],\]
which yields
\begin{equation}\label{eq:obviouslowerbound}
\mathbb{P}(T(\varnothing,1_n)\leq t)\geq(1+o(1))\cdot2^{-n}\cdot\left(\frac{t}{n}\right)^{\beta-1}\quad\text{as $n\to\infty$.}
\end{equation}
In the other direction, we prove the following inequality.

\begin{prop}\label{prop:kahn}
For every $n\in\mathbb{N}^*$, and for every $t>0$, we have
\begin{equation}\label{eq:kahn}
\mathbb{P}(T(\varnothing,1_n)\leq t)\leq2^{-n}\cdot\left(\frac{t}{n}\right)^{\beta-1}\cdot\exp\left[\sum_{k=1}^{n-1}\frac{k+1}{k^{\beta-1}}\cdot t^{\beta-1}\right].
\end{equation}
\end{prop}

This estimate is similar in spirit to \cite[Proposition 2.3]{moa}, which in turn was inspired by Kahn's proof of \cite[Theorem 5.1]{kahn}.

\begin{rem}
For $\beta>3$, since $\sum_{k\geq1}(k+1)\cdot k^{-(\beta-1)}<\infty$, we obtain
\[\mathbb{P}(T(\varnothing,1_n)\leq t)=O\left(2^{-n}\cdot\left(\frac{t}{n}\right)^{\beta-1}\right)\quad\text{as $n\to\infty$.}\]
This matches the order of magnitude of the obvious lower bound \eqref{eq:obviouslowerbound}.
Moreover, the estimate \eqref{eq:kahn} provides an alternative proof of the fact that the driving distance ball $\{x\in\mathbb{T}:T(\varnothing,x)\leq t\}$ is finite a.s.~for $\gamma\geq3$, by a first moment argument.
\end{rem}

\begin{proof}[Proof of Proposition \ref{prop:kahn}]
Fix $n\in\mathbb{N}^*$ and $t>0$.
For $k\in\llbracket1,n\rrbracket$, we denote by $e_k$ the edge between $1_{k-1}$ and $1_k$.
For every subset $E\subset\{e_1,\ldots,e_n\}$, let ${a(E)=\mathbb{P}\left(\sum_{e\in E}1/V_e\leq t\right)}$; and note that ${\mathbb{P}(T(\varnothing,1_n)\leq t)=a\{e_1,\ldots,e_n\}}$.
We claim that $a(\cdot)$ satisfies the following recursion: for every non-empty subset ${E\subset\{e_1,\ldots,e_n\}}$, we have
\begin{equation}\label{eq:reckahn}
a(E)\leq\sum_{\substack{\text{$F\subset E$ non-empty}\\\text{$F$ connected in $E$}}}\mathbb{P}\left(V_F\geq\frac{\#E}{t}\right)\cdot a(E\setminus F),
\end{equation}
where we denote by $V_F$ the speed of the fastest road that passes through every edge of $F$; and the sum is taken over all non-empty subsets $F\subset E$ which are \emph{connected in $E$}, in the sense that for every $i\leq j\in\llbracket1,n\rrbracket$ such that $e_i,e_j\in F$, we have $e_k\in F$ whenever $k\in\llbracket i,j\rrbracket$ is such that $e_k\in E$.
To prove \eqref{eq:reckahn}, let $E$ be a non-empty subset of $\{e_1,\ldots,e_n\}$.
On the event $\left(\sum_{e\in E}1/V_e\leq t\right)$, the fastest road of $\Pi$ that passes through at least one edge of $E$ must have speed at least $\#E/t$.
Denoting by $F$ the set of edges $e\in E$ that are traversed by this road, we obtain a non-empty subset $F\subset E$ which is connected in $E$, and such that $V_F^{E\setminus F}\geq\#E/t$, where $V_F^{E\setminus F}$ denotes the speed of the fastest road of $\Pi$ that passes through every edge of $F$ and no edge of $E\setminus F$.
This proves the inclusion
\[\left(\sum_{e\in E}\frac{1}{V_e}\leq t\right)\subset\bigcup_{\substack{\text{$F\subset E$ non-empty}\\\text{$F$ connected in $E$}}}\left(V_F^{E\setminus F}\geq\frac{\#E}{t}\,;\,\sum_{e\in E\setminus F}\frac{1}{V_e}\leq t\right).\]
By a union bound, this yields
\[a(E)\leq\sum_{\substack{\text{$F\subset E$ non-empty}\\\text{$F$ connected in $E$}}}\mathbb{P}\left(V_F^{E\setminus F}\geq\frac{\#E}{t}\,;\,\sum_{e\in E\setminus F}\frac{1}{V_e}\leq t\right).\]
Now, for each term in the sum, observe that the random variable $V_F^{E\setminus F}$ is independent of the random variables ${(V_e)_{e\in E\setminus F}}$.
Indeed, the former is measurable with respect to the restriction of $\Pi$ to the set of roads that pass through every edge of $F$ and no edge of $E\setminus F$, while the latter are measurable with respect to the restriction of $\Pi$ to the set of roads that pass through at least one edge of $E\setminus F$.
Thus, we obtain \eqref{eq:reckahn}:
\[\begin{split}
a(E)&\leq\sum_{\substack{\text{$F\subset E$ non-empty}\\\text{$F$ connected in $E$}}}\mathbb{P}\left(V_F^{E\setminus F}\geq\frac{\#E}{t}\right)\cdot\mathbb{P}\left(\sum_{e\in E\setminus F}\frac{1}{V_e}\leq t\right)\\
&\leq\sum_{\substack{\text{$F\subset E$ non-empty}\\\text{$F$ connected in $E$}}}\mathbb{P}\left(V_F\geq\frac{\#E}{t}\right)\cdot a(E\setminus F).
\end{split}\]
Upon reindexing the sum, we get
\[a(E)\leq\sum_{\substack{\text{$F\subsetneq E$}\\\text{$E\setminus F$ connected in $E$}}}\mathbb{P}\left(V_{E\setminus F}\geq\frac{\#E}{t}\right)\cdot a(F).\]
Since $a(\emptyset)=1$, iterating this inequality yields
\[a\{e_1,\ldots,e_n\}\leq\sum_{j=1}^n\sum_{\substack{\{e_1,\ldots,e_n\}=E_0\supsetneq\ldots\supsetneq E_j=\emptyset\\\text{$E_i\setminus E_{i+1}$ connected in $E_i$}}}\prod_{i=0}^{j-1}\mathbb{P}\left(V_{E_i\setminus E_{i+1}}\geq\frac{\#E_i}{t}\right).\]
Now, let us work on the summands above.
Using the inequality ${\mathbb{P}(\mathrm{Poisson}(\lambda)>0)\leq\lambda}$, we get
\[\begin{split}
\mathbb{P}\left(V_{E_i\setminus E_{i+1}}\geq\frac{\#E_i}{t}\right)&\leq\mu\{\ell\in\mathbb{L}:\text{$\ell$ passes through each edge of $E_i\setminus E_{i+1}$}\}\cdot\left(\frac{t}{\#E_i}\right)^{\beta-1}\\
&\leq2^{-\#(E_i\setminus E_{i+1})}\cdot\left(\frac{t}{\#E_i}\right)^{\beta-1}.
\end{split}\]
We deduce that
\[\prod_{i=0}^{j-1}\mathbb{P}\left(V_{E_i\setminus E_{i+1}}\geq\frac{\#E_i}{t}\right)\leq\prod_{i=0}^{j-1}\left(2^{-\#(E_i\setminus E_{i+1})}\cdot\left(\frac{t}{\#E_i}\right)^{\beta-1}\right)=2^{-n}\cdot\left(\frac{t}{n}\right)^{\beta-1}\cdot\prod_{i=1}^{j-1}\left(\frac{t}{\#E_i}\right)^{\beta-1}.\]
At this point, we have obtained
\begin{equation}\label{eq:propquickco}
\mathbb{P}(T(\varnothing,1_n)\leq t)\leq2^{-n}\cdot\left(\frac{t}{n}\right)^{\beta-1}\cdot\sum_{j=1}^n\sum_{\substack{\{e_1,\ldots,e_n\}=E_0\supsetneq\ldots\supsetneq E_j=\emptyset\\\text{$E_i\setminus E_{i+1}$ connected in $E_i$}}}\prod_{i=1}^{j-1}\left(\frac{t}{\#E_i}\right)^{\beta-1},
\end{equation}
and the remaining work is purely combinatorial.
For each $j\in\llbracket1,n\rrbracket$, grouping the terms according to $k_i=\#E_i$, we can compute exactly:
\[\begin{split}
\sum_{\substack{\{e_1,\ldots,e_n\}=E_0\supsetneq\ldots\supsetneq E_j=\emptyset\\\text{$E_i\setminus E_{i+1}$ connected in $E_i$}}}\prod_{i=1}^{j-1}\left(\frac{t}{\#E_i}\right)^{\beta-1}&=\sum_{n=k_0>\ldots>k_j=0}(k_1+1)\cdot\ldots\cdot(k_{j-1}+1)\cdot\prod_{i=1}^{j-1}\left(\frac{t}{k_i}\right)^{\beta-1}\\
&=\sum_{n=k_0>\ldots>k_j=0}\prod_{i=1}^{j-1}\left(\frac{k_i+1}{k_i^{\beta-1}}\cdot t^{\beta-1}\right).
\end{split}\]
Indeed, given any subset $E_i\subset\{e_1,\ldots,e_n\}$ with cardinality $k_i$ and any integer $k_{i+1}\in\llbracket1,k_i\llbracket$, there are $(k_{i+1}+1)$ ways of choosing a subset $E_{i+1}\subset E_i$ with cardinality $k_{i+1}$ such that $E_i\setminus E_{i+1}$ is connected in $E_i$.
The above equality leads to the upper bound:
\[\begin{split}
\sum_{\substack{\{e_1,\ldots,e_n\}=E_0\supsetneq\ldots\supsetneq E_j=\emptyset\\\text{$E_i\setminus E_{i+1}$ connected in $E_i$}}}\prod_{i=1}^{j-1}\left(\frac{t}{\#E_i}\right)^{\beta-1}&\leq\frac{1}{(j-1)!}\cdot\sum_{1\leq k_1,\ldots,k_{j-1}\leq n-1}~\prod_{i=1}^{j-1}\left(\frac{k_i+1}{k_i^{\beta-1}}\cdot t^{\beta-1}\right)\\
&=\frac{1}{(j-1)!}\cdot\left(\sum_{k=1}^{n-1}\frac{k+1}{k^{\beta-1}}\cdot t^{\beta-1}\right)^{j-1}.
\end{split}\]
Plugging this into \eqref{eq:propquickco} and summing over $j$, we obtain \eqref{eq:kahn}.
\end{proof}

\subsection{Open questions}\label{subsec:open}

To conclude this paper, let us state some natural open questions raised by our results.

\begin{itemize}
\item The proof of Theorem \ref{thm:explosion} falls short of describing what happens at $\beta=2$.
It might be the case that there is no explosion; but note that in contrast with the result of Proposition \ref{prop:noexplosion}, when $\beta=2$, for every $t>0$, we have
\[\begin{split}
\mathbb{E}\left[\#\{x\in\mathbb{T}:T(\varnothing,x)\leq t\}\right]&\geq\mathbb{E}\left[1+2\cdot\lfloor t\cdot V_\varnothing\rfloor\right]\\
&\geq\mathbb{E}[t\cdot V_\varnothing]\\
&=t\cdot\int_0^\infty\left(1-\exp\left[-3/4\cdot v^{-1}\right]\right)\mathrm{d}v=\infty.
\end{split}\]

\item The set
\[\left\{(x_n)_{n\in\mathbb{N}}\in\partial\mathbb{T}:\sum_{n\geq1}T(x_{n-1},x_n)<\infty\right\}\]
has measure $0$ a.s, with respect to the natural Borel measure on $\partial\mathbb{T}$ introduced in Section \ref{sec:mu}; on the other hand, this set must be dense in $\partial\mathbb{T}$ as soon as the explosion event ${(T(\varnothing,\partial\mathbb{T})<\infty)}$ is realised.
In that case, it would be interesting to compute its Hausdorff dimension, with respect to the distance $d$ on $\partial\mathbb{T}$ introduced in Section \ref{sec:mu}.

\item Although we fail to obtain a matching upper bound for $\beta\in{]2,3]}$, it seems plausible that the obvious lower bound \eqref{eq:obviouslowerbound} gives the right order of magnitude for the bounded driving distance probability in the whole phase $\beta>2$.

\item The results presented in this paper should hold more generally in the $d$-regular tree for all $d\geq2$.

\item We expect that a result similar to Theorem \ref{thm:explosion} holds for the driving distance problem in the hyperbolic plane.
We intend to investigate this in a forthcoming paper.
\end{itemize}

\bibliographystyle{siam}
\bibliography{biblio.bib}

\begin{thebibliography}{10}

\bibitem{aldous}
{\sc D.~Aldous}, {\em {Scale-invariant random spatial networks}}, Electronic
  Journal of Probability, 19 (2014), pp.~1--41.

\bibitem{visibility}
{\sc I.~Benjamini, J.~Jonasson, O.~Schramm, and J.~Tykesson}, {\em {Visibility
  to infinity in the hyperbolic plane, despite obstacles}}, Latin American
  Journal of Probability and Mathematical Statistics, 6 (2009), pp.~323--342.

\bibitem{moa}
{\sc G.~Blanc}, {\em {Fractal properties of Aldous--Kendall random metric}},
  arXiv, 2207.03349 (2022).

\bibitem{kahn}
{\sc J.~Kahn}, {\em {Improper Poisson line process as SIRSN in any dimension}},
  The Annals of Probability, 44 (2016), pp.~2694 -- 2725.

\bibitem{kendall}
{\sc W.~S. Kendall}, {\em {From random lines to metric spaces}}, The Annals of
  Probability, 45 (2017), pp.~469 -- 517.

\bibitem{lastpenrose}
{\sc G.~Last and M.~Penrose}, {\em Lectures on the Poisson Process}, Institute
  of Mathematical Statistics Textbooks, Cambridge University Press, 2017.

\bibitem{lyonsperes}
{\sc R.~Lyons and Y.~Peres}, {\em Probability on Trees and Networks}, Cambridge
  Series in Statistical and Probabilistic Mathematics, Cambridge University
  Press, 2017.

\bibitem{jeulin}
{\sc A.~Matsumoto and K.~Yano}, {\em On a Zero-One Law for the Norm Process of
  Transient Random Walk}, Séminaire de Probabilités XLIII, Springer, 2011,
  pp.~105--126.

\bibitem{pemantleperes}
{\sc R.~Pemantle and Y.~Peres}, {\em {Domination between trees and application
  to an explosion problem}}, The Annals of Probability, 22 (1994),
  pp.~180--194.

\bibitem{stochintgeo}
{\sc R.~Schneider and W.~Weil}, {\em Stochastic and Integral Geometry},
  Probability and Its Applications, Springer, 2008.

\bibitem{sidoravicius}
{\sc V.~Sidoravicius and A.-S. Sznitman}, {\em Percolation for the vacant set
  of random interlacements}, Communications on Pure and Applied Mathematics, 62
  (2009), pp.~831--858.

\bibitem{sznitman}
{\sc A.-S. Sznitman}, {\em Vacant set of random interlacements and
  percolation}, Annals of Mathematics, 171 (2010), pp.~2039--2087.

\end{thebibliography}

\end{document}